\documentclass[11pt,final]{amsart}
\usepackage{amsmath}
\usepackage{amsfonts}
\usepackage{amssymb}
\usepackage{graphicx} 

\usepackage[color,notref,notcite]{showkeys}  
\definecolor{labelkey}{rgb}{0.6,0,1}

\newtheorem{theorem}{Theorem}[section]

\newtheorem{lemma}[theorem]{Lemma}

\theoremstyle{definition}

\newtheorem{remark}{Remark}

\newcommand{\eps}{\varepsilon}

\usepackage{constants}
\newconstantfamily{ctrcst}{symbol=C,reset={bibunit}}
\def\ctel#1{\ensuremath{\Cl[ctrcst]{#1}}}
\def\cter#1{\ensuremath{\Cr{#1}}}

\def\act#1#2#3#4{\langle#1,#2\rangle_{#3,#4}}  

\def\benum{\begin{enumerate}}
\def\benumc{\begin{enumerate}}
\def\be{\begin{equation}}

\def\dsp{\displaystyle}

\def\eenum{\end{enumerate}}
\def\enumc{\end{enumerate}}
\def\ee{\end{equation}}

\def\grad{\nabla}

\newcommand{\hunz}{{H^1_0(\O)}}

\def\hunpo{H^{-1}(\O)}

\def\nnnr{{n \in \mathbb{N}^\star}}
\def\norm#1#2{\left\| #1 \right\|_{#2}}

\def\O{\Omega}

\def\R{\mathbb{R}}
\def\N{\mathbb{N}}

\def\trm{\textrm}

\def\vphi{\varphi}

\newcommand{\ldo}{{L^2(\O)}}

\newcommand{\dx}{\ \mathrm{d}x}
\newcommand{\dt}{\ \mathrm{d} t}
\newcommand{\ds}{\ \mathrm{d} s}

\newcommand{\nnn}{{n\in\mathbb{N}}}

\newcommand{\nti}{{n\to +\infty}}
\title[Convergence  proof for the Stefan problem]
{A new  convergence  proof for \\ approximations of the Stefan problem}

\author{Robert Eymard and Thierry Gallou\"et}

\begin{document}


\date{}

\begin{abstract}
We consider the Stefan problem, firstly with regular data and secondly with irregular data. In both cases is given a proof for the convergence of an approximation obtained by regularising the problem. These proofs are based on weak formulations and  on compactness results in some Sobolev spaces with negative exponents. 
\end{abstract}
\maketitle

%
%
\section{Introduction}\label{sec-pb}

Let $\Omega$ be an open bounded domain of $\mathbb{R}^d$ ($d \ge 1$), $0<T<+\infty$. Let $\vphi$ be a 
Lipschitz continuous function from $\mathbb{R}$ to $\mathbb{R}$.
We assume that $\vphi$ is a nondecreasing function (but $\vphi$ is not necessarily increasing, it can be constant  on an interval which has a positive Lebesgue measure).
We consider the wellknown Stefan problem, which is a free boundary problem describing the evolution of the boundary between two phases of a material undergoing a phase change.
The model problem reads
\begin{subequations}
\label{stefan-fort}
 \begin{align}
&\partial_t u  - \Delta(\vphi(u))=f \mbox{ on } \Omega\times]0,T[, \label{eq:stefan-fort} \\ 
&u(x,0) = u_0(x)  \mbox{ on } \Omega, \\
&u(x,t) = 0 \mbox{ on } \partial\Omega\times]0,T[
\end{align}
\end{subequations}
In the seminal paper \cite{alt_luckhaus}, Alt and Luckhaus prove the existence of a weak solution to the problem \eqref{stefan-fort} in the case where $f \in L^2(]0,T[ , L^2(\Omega))$, $u_0 \in L^2(\O)$.
This proof is first based on the regularisation of the problem \eqref{stefan-fort} by introducing in Equation \eqref{eq:stefan-fort} an additional term  $- (1/n) \Delta u$, with $n>0$ (the existence (and uniqueness) of a  solution $u_n$ to such a regularised problem is then classical). 
Compactness arguments on the solution $u_n$ to the regularised problem are then derived in the following way:

\begin{itemize}
 \item It is shown that, up to some subsequence, there exists $v\in L^{2}(]0,T[,H^1_0(\O))$ such that  $\vphi(u_n) \to v$   in $L^{2}(]0,T[,\ldo)$ and weakly in $L^{2}(]0,T[,H^1_0(\Omega))$.
 \item It is also shown that, for this subsequence, there exists $u\in L^{2}(]0,T[,\ldo)$ such that  $u_n \to u$  weakly  in $L^{2}(]0,T[,\ldo)$.
 \item Minty's trick is then used to obtain that $v = \vphi(u)$, by weak/strong convergence in $L^{2}(]0,T[,\ldo)$.
\end{itemize}

In Section \ref{sec:regrhs}, we give an alternate proof of this convergence result, stated in Theorem \ref{thm:reg}. 
The originality of the new proof lies in the proof of compactness in a negative exponent Sobolev space, namely $L^2(]0,T[, H^{-1}(\Omega))$ in the case where $f \in L^2(]0,T[, L^2(\Omega))$. 
Indeed, the passage to the limit as $\nti$ is then based on the convergence of the sequence $(u_n)_\nnn$ in $L^2(]0,T[, H^{-1}(\Omega))$ and the weak convergence of the sequence $(\vphi(u_n))_\nnn$ in $L^2(]0,T[, H^1_0(\Omega))$. 
Negative  exponent Sobolev spaces were first used in compactness arguments for a turbulence problem in \cite{gal2012comp}, and also in \cite{moussa2016vari} in order to obtain some generalised Aubin-Lions lemmas.
In \cite{and2017nonlinar}, a related method presented as compensated compactness, is introduced for monotone graph problems.

In  Section \ref{sec:irregrhs}, we consider the case of irregular data $f\in L^{1}(]0,T[,L^1(\O))$ and  $u_0\in L^1(\O)$. 
Note that, in this case, the time compactness cannot be handled by the method used in \cite{alt_luckhaus}, as noticed in \cite{and2017nonlinar}. 
In \cite{igbida2010renorm},  the authors prove the existence and uniqueness of a solution to the Stefan problem with $L^1$ data. The time compactness is obtained by introducing  a time regularisation of the truncations of $\vphi(u^{(n)})$.
The aim of this second part is to obtain a similar result without this additional time regularisation. Instead we use, as in the first part of this paper, some compactness results in negative exponent Sobolev spaces on the solutions of approximate problems.
Indeed, we introduce sequences of functions $(f^{(n)})_\nnn \subset L^{2}(]0,T[,\ldo)$ and $(u_0^{(n)})_\nnn \subset L^2(\O)$ respectively approximating in $L^1$ the functions $f$ and $u_0$. 
Then, denoting $u^{(n)}$ the solution of \eqref{stefan} with data $f^{(n)}$ and $u_0^{(n)}$ instead of $f$ and $u_0$, we study the convergence of $u^{(n)}$ and $\vphi(u^{(n)})$ as $\nti$.

For such irregular data, we only prove this convergence in the case $d=2$ or $d=3$ under Hypothesis \eqref{hyp:surlin}, stating that the function $\vphi$ dominates some linear function at $\infty$: this hypothesis  is highly used in the proof of Lemma \ref{lem:estimhun}, providing some estimates used in the proof of Theorem \ref{thm:irreg}.

\smallskip

We recall that in a Banach space $E$ equipped with a norm $\Vert \cdot \Vert_E$, a sequence $(u_n)_\nnn \subset E$ is said to converge to $u \in E$ if $\Vert u_n - u\Vert_E\to 0$ as $\nti$, while it is said to weakly converge to $u \in E$ if for any continuous linear form $T\in E'$,  $T(u_n) \to  T(u)$ as $\nti$.
A sequence $(T_n)_\nnn \subset E'$ is said to $\star$-weakly converge to $T \in E'$ if for any $u \in E$ $T_n(u) \to  T(u)$ as $\nti$.
If $E=L^p(\Omega)$, where $1 \le p <+\infty$ and $\Omega$ is an open set of $\mathbb{R}^d$, the space $E'$ is identified to $L^q(\Omega)$, $q=p/(p-1)$.
For $T>0$ and $E=L^1(]0,T[, L^2(\Omega))$, we also identify the space $L^\infty(]0,T[, L^2(\Omega))$ with $E'$.

\section{The case of regular data}\label{sec:regrhs}

\subsection{Weak formulations and convergence theorem}

Let $\O$ be an open bounded subset of $\R^d$ with $d\in\N^\star$, and let $T>0$. We consider the following weak sense for a solution to \eqref{stefan-fort}.

\begin{subequations}\label{stefan}
 \begin{align}
	&\left\{\begin{array}{l}
	u \in L^{\infty}(]0,T[,L^2(\Omega)), \partial_t u  \in L^2(]0,T[, H^{-1}(\Omega)), u \in C([0,T], H^{-1}(\Omega)),\\
	\vphi(u) \in L^{2}(]0,T[,H^1_0(\Omega)), \\
	\dsp \int_0^T \left\langle \partial_t u(s) , v(s)\right\rangle_{H^{-1},H^{1}_0 } \ds + \int_0^T \int_\Omega \grad \vphi(u(x,s))  \cdot \nabla v(x,s) \dx \ds  \\
	 \hfill =\dsp \int_0^T \int_\Omega   f(x,s) v(x,s)  \dx \ds , \qquad \forall v \in L^2(]0,T[, H^{1}_0(\Omega)), \\
	\end{array}\right. \label{eq:stefan}
	\\
	& \quad u(\cdot,0 ) = u_0,\label{eq:stefanz}
\end{align}
\end{subequations}
where we denote by  $u(s)$ (resp. $v(s)$) the function $x \mapsto u(x,s)$ (resp. $v(x,s)$).
Since $\ldo$ is identified with $\ldo'$, one has $\hunz \subset \ldo = \ldo' \subset \hunpo$. 
The function $\partial_t u$ is the weak time derivative of $u$ (see Definition 4.22 of \cite{edp-gh})
The fact that $u \in L^2(]0,T[ , H^1_0 (\Omega))$ and $\partial_t u  \in L^2 (] 0, T[ , H^{-1} (\Omega))$ give
$u \in C([0,T],\ldo)$ (see Lemma 4.26 of  \cite{edp-gh}). The function $u$ is defined for all $t \in [0,T]$,
which gives sense to the initial condition $u(0)=u_0$ a.e..
 
We prove in  Section \ref{sec:regrhs} the convergence to a solution of Problem \eqref{stefan} of a solution given by \eqref{eq-pf} below, which is obtained from Problem \eqref{stefan} by the addition of vanishing diffusion $-(1/n)\Delta u$. Indeed, as recalled in the introduction, a classical result for parabolic equations gives, for all $n>0$, the existence of a (unique) solution $u_n$ to the regularisation of \eqref{stefan} in the sense specified in the introduction.
Furthermore, this solution $u_n$ belongs to $L^2(]0,T[, H^1_0 (\Omega)) \cap C([0,T],\ldo)$, and satisfies
\be
\left\{\begin{array}{llllll}
u_n \in L^\infty(]0,T[, H^1_0 (\Omega)), \; \partial_t u_n  \in L^2(]0,T[, H^{-1}(\Omega)), \\
\displaystyle \int^T_0 \langle \partial_t u_n(s) , v(s)\rangle_{H^{-1}, H^1_0} \ds \\ \hfill \dsp+ \int^T_0  \int_\Omega (\vphi'(u_n(x,s))+\frac 1 n) \nabla u_n(x,s) \cdot \nabla v(x,s) \dx  \ds
\\ \dsp \hfill = \int^T_0 (\int_\O f(s,x) v(s,x) \dx) \ds,
\qquad \forall v \in L^2(]0,T[, H^1_0(\Omega)),\\
u(0) = u_0 \ \trm{ a.e.}.\end{array}\right.
\label{eq-pf}
\ee
We recall that $\grad \vphi(v)=\vphi'(v) \grad v$ a.e. if $v \in L^2(]0,T[, H^{1}_0(\Omega))$ (see Lemma 4.31 in \cite{edp-gh}).
Existence of $u_n$ can be proven using Schauder's Theorem  and the resolution of linear parabolic equations
by Faedo-Galerkin's method. This method is detailed in \cite[Theorem 4.28]{edp-gh} (for the heat equation) and in \cite[Exercise 4.5]{edp-gh} (corrected) for a more general
diffusion operator. \cite[Exercise 4.6]{edp-gh} (corrected) gives the existence of $u_n$ using Schauder's Theorem.

We state the convergence of $u_n$ in the following theorem.

\begin{theorem}[Convergence of the regularisation method]\label{thm:reg} Let $\O$ be an open bounded subset of $\R^d$, with $d\in\N^\star$ and let $T>0$. Let $\vphi$ be a nondecreasing Lipschitz continuous function. Let  $f\in L^2(]0,T[,L^2(\O))$ and  $u_0\in L^2(\O)$ be given. 
 Let $(u_n)_\nnnr$ be the solution of \eqref{eq-pf} for all $\nnnr$. Then there exists $u$, solution of \eqref{stefan}, such that, as $n\to +\infty$ up to some subsequence,
 \begin{itemize}
  \item $u_n$ converges to $u$ in $C([0,T],H^{-1}(\O))$ and $\star$-weakly in $L^\infty(]0,T[,L^2(\O))$,
  \item $\vphi(u_n)$ converges to $\vphi(u)$ in $L^2(]0,T[,H^1_0(\O))$.
 \end{itemize}

\end{theorem}

The remaining of  Section \ref{sec:regrhs} is dedicated to the proof of Theorem \ref{thm:reg}

\subsection{Estimates on the approximate solution}\label{sec-eap}

Let $\Phi$ be a primitive of $\vphi$ and let $n>0$.
Since $u_n \in L^2(]0,T[, H^1_0 (\Omega))$ and $\partial_t u_n  \in L^2(]0,T[, H^{-1}(\Omega))$, Lemma 4.31 of \cite{edp-gh} gives $\Phi(u_n) \in C([0,T],L^1(\Omega))$ and, for $0 \le t_1 < t_2 \le T$,
\[
\int_\Omega \Phi(u_n(t_2)) \dx - \int_\Omega \Phi(u_n(t_1)) \dx=\int_{t_1}^{t_2}  \langle \partial_t u_n(s) , \vphi(u_n(s))\rangle_{H^{-1}, H^1_0}\ds.
\]
Then taking $v=\vphi(u_n)$ in \eqref{eq-pf}, we obtain an estimate on $\Phi(u_n)$ in $C([0,T],L^1(\Omega))$ and an estimate on $\vphi(u_n)$ in $ L^2(]0,T[, H^1_0 (\Omega))$, that is the existence of $C$, only depending on $f$ and $u_0$ such that, for all $t \in [0,T]$,
\be
\norm{\Phi(u_n(t))}{L^1(\Omega)} \le C,
\label{est-cl1}
\ee
and
\be
\norm{\vphi(u_n)}{L^2(]0,T[, H^1_0(\Omega))} \le C \hbox{ and }\norm{u_n}{L^2(]0,T[, H^1_0(\Omega))} \le C \sqrt{n}.
\label{est-h10}
\ee
Using \eqref{est-h10} and \eqref{eq-pf} once again yield an estimate on $\partial_t u_n$ in $L^2(]0,T[,H^{-1}(\Omega))$ that is expressed by
\be
\norm{\partial_tu_n}{L^2(]0,T[, H^{-1}(\Omega))} \le C.
\label{est-hm1}
\ee
We will now also obtain an estimate on $u_n$ in $C([0,T],L^2(\Omega))$ thanks to the fact that $f \in L^2(]0,T[ ,L^2(\Omega))$
We first recall that $u_n \in C([0,T],L^2(\Omega))$ and taking $v=u_n$ in \eqref{eq-pf} leads to (with Lemma 4.26 of \cite{edp-gh}), for all $t \in [0,T]$,
\begin{multline*}
\frac 1 2 \norm{u_n(t)}{L^2(\Omega)}^2 - \frac 1 2 \norm{u_0}{L^2(\Omega)}^2 \le \int_0^t \int_\O f(x,s) u_n(x,s) \dx\ds
\\ 
\le \int_0^t \norm{f(s)}{\ldo}\norm{u_n(s)}{\ldo}\ds
\\ \le \frac 1 2 \int_0^T \norm{f(s)}{\ldo}^2\ds +  \frac 1 2\int_0^t \norm{u_n(s)}{\ldo}^2\ds,
\end{multline*}
which leads to an estimate on $\norm{u_n(t)}{L^2(\Omega)}^2$ by the classical Gronwall technique.
Therefore, here also, we obtain, for all $t \in [0,T]$,
\be
\norm{u_n(t)}{L^2(\Omega)} \le C.
\label{est-ld}
\ee

\subsection{Passing to the limit}

Estimate \eqref{est-h10} allows us to assume that, up to a subsequence, still denoted $(u_n)_\nnn$, $\vphi(u_n) \to \zeta$ weakly in 
$L^{2}(]0,T[,H^1_0(\Omega))$.
Furthermore, since $\ldo$ is compactly embedded in $H^{-1}(\O)$, we may also assume, up to a subsequence, still denoted $(u_n)_\nnn$,
that $u_n \to u$  in $L^{2}(]0,T[,H^{-1}(\Omega))$ (thanks to \eqref{est-hm1}).
This property is given for instance by Theorem 4.42 of  \cite{edp-gh}, with $B=Y=H^{-1}(\O)$ and $X=\ldo$.
With estimate \eqref{est-hm1} and linearity of the operator $\partial_t$ we also have
$\partial_t u_n \to \partial_t u$ weakly in $L^{2}(]0,T[,H^{-1}(\Omega))$.
We also observe that, owing to \eqref{est-h10}, the inequality
\[
 \big| \frac 1 n \int^T_0  \int_\Omega  \nabla u_n(x,s) \cdot \nabla v(x,s) \dx  \ds\big| \le \frac C {\sqrt{n}}\norm{v}{L^2(]0,T[, H^1_0(\Omega))}
\]
implies that
\[
 \lim_{n\to\infty}\frac 1 n \int^T_0  \int_\Omega  \nabla u_n(x,s) \cdot \nabla v(x,s) \dx  \ds = 0.
\]

Thanks to Estimate \eqref{est-ld}, we also have $u_n \to u$  weakly in $L^{2}(]0,T[,\ldo)$ and $u \in L^\infty(]0,T[,\ldo)$.

Since $u \in L^2(]0,T[,\ldo) \subset L^2(]0,T[,H^{-1}(\Omega))$ and $\partial_t u \in L^{2}(]0,T[,H^{-1}(\Omega))$, the function $u$ belongs to $C([0,T],H^{-1}(\O)$ and even to  $C^{1/2}([0,T],H^{-1}(\O)$ (see, for instance Lemma 4.25 of \cite{edp-gh}).

It is now possible to pass to the limit in \eqref{eq-pf}, it gives, for all $v \in L^2(]0,T[, H^1_0(\Omega))$
\begin{multline}
 \int^T_0 \langle \partial_t u(s) , v(s)\rangle_{H^{-1}, H^1_0}\ds + \int^T_0  \int_\Omega  \nabla \zeta(x,s) \cdot \nabla v(x,s) \dx  \ds 
\\ \dsp \hfill = \int^T_0 (\int_\O f(s,x) v(s,x) \dx) \ds,
\label{eq-pfl}
\end{multline}
In order to prove that $u$ is a solution of \eqref{stefan} it remains to prove that $\zeta=\vphi(u)$ a.e. and that $u(0)=u_0$.
This is done in the following sections.

\subsection{Minty trick with compactness in $L^{2}(]0,T[,H^{-1}(\Omega))$}\label{sec:minty}
\begin{lemma}[Minty trick]\label{lem:mintyneg} Let $\O$ be an open bounded subset of $\R^d$ with $d\in\N^\star$ and let $T>0$.
 Let $(u_n)_{\nnnr}$ be a bounded sequence of elements of $L^{2}(]0,T[,L^r(\O))$ for a given $r$, with $r \ge \frac {2d}{d+2}$ if $d\ge 3$, $r>1$ if $d=2$ and $r=1$ if $d=1$, and $u\in L^{2}(]0,T[,L^r(\O))$ such that $u_n$ converges to $u$ in $L^{2}(]0,T[,H^{-1}(\Omega))$ as $n\to\infty$. Let $\vphi$ be a non decreasing Lipschitz continuous function, such that $(\vphi(u_n))_{\nnnr}$ is bounded in $L^{2}(]0,T[,H^1_0(\Omega))$.
 
 Then $(\vphi(u_n))_{\nnnr}$ converges to $\vphi(u)$ in $L^{2}(]0,T[,L^p(\O))$ for any $p\in [1,2d/(d-2)[$ if $d\ge 3$ and any $p\in [1,+\infty[$ if $d\le 2$, and weakly in $L^{2}(]0,T[,H^1_0(\Omega))$.
\end{lemma}
\begin{remark}
 The proof of this lemma could be obtained by applying \cite[Proposition 1.4]{and2017nonlinar}. According to the objectives of this paper, we provide below a proof based on the direct use of compactness properties in negative  exponent Sobolev spaces.
\end{remark}

\begin{proof}
Let us first observe that $u_n\to u$ weakly in $L^{2}(]0,T[,L^r(\O))$, since there exists subsequences of $(u_n)_{\nnnr}$ which converge weakly in  $L^{2}(]0,T[,L^r(\O))$ but the limits of these subsequences are necessarily equal to $u$ (since
$u_n$ converges to $u$ in $L^{2}(]0,T[,H^{-1}(\Omega))$), then the whole sequence $(u_n)_\nnn$ converges to $u$ weakly in $L^{2}(]0,T[,L^r(\O))$.

\medskip

For any $k>0$, we denote by $T_k$ the truncation function defined by
\begin{equation}\label{eq:deftrunc}
\forall s\in\R,\ T_k(s) = \min(k,|s|){\rm sign}(s).
\end{equation}
We then remark that the sequence $(T_k(\vphi(u_n)))_{\nnnr}$ is also bounded in $L^{2}(]0,T[,H^1_0(\Omega))$ (again applying Lemma 4.31 in \cite{edp-gh}).
We denote by $(u_n)_{\nnnr}$ a subsequence of $(u_n)_{\nnnr}$ such that $T_k(\vphi(u_n)) \to \zeta_k$ weakly in 
$L^{2}(]0,T[,H^1_0(\Omega))$, where $\zeta_k\in L^{2}(]0,T[,H^1_0(\Omega))$. Since $u_n \to u$  in $L^{2}(]0,T[,H^{-1}(\Omega))$, 
\begin{multline}
\int_0^T \int_\O T_k(\vphi(u_n(x,t))) (u_n(x,t)-u(x,t)) \dx\dt \\ = \int_0^T \act{u_n(t)-u(t)}{T_k(\vphi(u_n(t)))}{H^{-1}}{H^1_0}\dt\\
\to \int_0^T \act{0}{\zeta_k(t)}{H^{-1}}{H^1_0}\dt
=0
\textrm{~as~} \nti.\label{eq:cvphiun}
\end{multline}
Let $\Theta_n$ be  the nonnegative function defined, for a.e. $(x,t)\in \O\times]0,T[$, by
\[
 \Theta_n(x,t) = \big(T_k(\vphi(u_n(x,t)))-T_k(\vphi(u(x,t)))\big) \big(u_n(x,t)-u(x,t)\big).
\]
Writing, owing to the nonnegativity of $\Theta_n$,
\begin{multline*}
\Vert \Theta_n\Vert_1 =\int_0^T \int_\O \Theta_n(x,t)  \dx\dt=
\int_0^T \int_\O T_k(\vphi(u_n(x,t)))(u_n(x,t) - u(x,t)) \dx\dt \\ -  \int_0^T \int_\O T_k(\vphi(u(x,t)))(u_n(x,t) - u(x,t))\dx\dt,
\end{multline*}
we obtain, using \eqref{eq:cvphiun} for the first term and the weak convergence of $u_n$ to $u$ in $L^{2}(]0,T[,L^r(\O))$ for the second term, that
\[
 \lim_{n\to \infty}\Vert \Theta_n\Vert_1 = 0.
\]
Since we have
\begin{multline*}
 \int_0^T\int_\O \big(T_k(\vphi(u_n))-T_k(\vphi(u))\big)^2\dx \dt \\
 \le L_\vphi \int_0^T\int_\O \big(T_k(\vphi(u_n))-T_k(\vphi(u))\big)\big(u_n-u\big)\dx \dt = L_\vphi \Vert \Theta_n\Vert_1,
\end{multline*}
where $L_\vphi$ is given by the Lipschitz continuity of $\vphi$, 
we obtain that $T_k(\vphi(u_n))$ converges to $T_k(\vphi(u))$ in $L^2(]0,T[,L^2(\O))$. By uniqueness of the limit, this convergence holds for the whole sequence 
$(u_n)_{\nnnr}$.

\medskip

We therefore get that, for all $k>0$, $T_k(\vphi(u_n))$ converges to $T_k(\vphi(u))$ in $L^2(]0,T[,L^2(\O))$. Applying Lemma \ref{lem:cvtkstrong} in the appendix with $p = \min(r,2)$, this proves that
$\vphi(u_n)$ converges to $\vphi(u)$ in  $L^1(]0,T[,L^1(\Omega))$. But, since the sequence $(\vphi(u_n))_{\nnnr}$ is bounded in $L^{2}(]0,T[,H^1_0(\Omega))$ (and then admits weakly converging subsequences in $L^{2}(]0,T[,H^1_0(\Omega))$ whose the limit is necessarily, at least in the distribution sense, the same  as the limit in $L^1(]0,T[,L^1(\Omega))$), this implies that $\vphi(u) \in L^{2}(]0,T[,H^1_0(\Omega))$ and $\vphi(u_n) \to \vphi(u)$ weakly in 
$L^{2}(]0,T[,H^1_0(\Omega))$.
By Sobolev inequality, we therefore have $\vphi(u_n)$ and $\vphi(u)$ bounded in $L^{2}(]0,T[,L^p(\O))$ for  $p=2d/(d-2)$ if $d\ge 3$ and for any $p\in [1,+\infty[$ if $d\le 2$. By interpolation, this implies that $\vphi(u_n)\to\vphi(u)$  in $L^{2}(]0,T[,L^p(\O))$ for any $p\in [1,2d/(d-2)[$ if $d\ge 3$ and any $p\in [1,+\infty[$ if $d\le 2$.
\end{proof}

\subsection{Initial condition and conclusion of the convergence proof}\label{sec:inihmun}

The above results prove that $u$ is solution of  \eqref{eq:stefan}, that is $u$ satisfies
 \begin{equation*}
	\left\{\begin{array}{l}
	u \in L^{\infty}(]0,T[,L^2(\Omega)), \partial_t u  \in L^2(]0,T[, H^{-1}(\Omega)), u \in C([0,T], H^{-1}(\Omega))),\\
	\vphi(u) \in L^{2}(]0,T[,H^1_0(\Omega)) \\
	\dsp \int_0^T \left\langle \partial_t u(s) , v(s)\right\rangle_{H^{-1},H^{1}_0 } \dt + \int_0^T \int_\Omega \grad \vphi(u(x,s))  \cdot \nabla v(x,s) \dx \dt  \\
	 \hfill =\dsp \int_0^T \int_\Omega   f(x,s) v(x,s)  \dx \dt , \qquad \forall v \in L^2(]0,T[, H^{1}_0(\Omega)).
	\end{array}\right. \label{stefaninc}\
\end{equation*}
It remains to prove \eqref{eq:stefanz}, that is $u(0)=u_0$. For this purpose, we will prove below that the sequence $(u_n)_\nnnr$ is relatively compact in $C([0,T],H^{-1}(\O))$.
Indeed, if this compactness is proven, it exists $w \in C([0,T], H^{-1}(\O))$ and a subsequence, still denoted as 
$(u_n)_\nnnr$, such that $u_n(t) \to w(t)$ in $H^{-1}(\O)$ uniformly with respect to $t \in [0,T]$ (and then also in $L^2(]0,T[,H^{-1}(\O))$).
In particular $w(0)=\lim_{\nti} u_n(0)=u_0$.
But, since we already know that 
$u_n \to u$ in $L^2(]0,T[,H^{-1}(\O))$, by uniqueness of the limit,  $u=w$ a.e. on $]0,T[$ and $u(t)=w(t)$ for all $t \in [0,T]$ since $u$ and $w$ are continuous on $[0,T]$.
Finally we obtain $u(0)=w(0)=u_0$.

It remains to show that the sequence $(u_n)_\nnnr$ is relatively compact
in \\ $C([0,T],H^{-1}(\O))$.
Using Ascoli's Theorem, it is enough to prove:
\benum
\item For all $t \in [0,T]$, $(u_n(t))_\nnnr$ is relatively compact in $H^{-1}(\O)$.
\item $\norm{u_n(t)-u_n(s)}{H^{-1}} \to 0$, as $s \to t$, uniformly with respect to $\nnnr$ (and for all $t \in [0,T]$).
\eenum
The second item is a consequence of $\partial_t u_n \in L^1([0,T], H^{-1}(\O))$ since Lemma 4.25 of \cite{edp-gh} gives
for all  $t_1,t_2 \in [0,T]$, $t_1 > t_2$ and all $\nnnr$,
\[
u_n(t_1)-u_n(t_2)=\int_{t_2}^{t_1} \partial_t u_n(s)\ds,
\]
and then
\begin{multline*}
\norm{u_n(t_1)-u_n(t_2)}{H^{-1}} \le \int_{t_2}^{t_1} \norm{\partial_t u_n(s)}{H^{-1}} \ds
\\
\le \big( \int_0^T \norm{\partial_t u_n(s)}{H^{-1}}^2\ds\big)^{\frac 1 2}
\sqrt{t_1-t_2}
\\
 \le \norm{\partial_t u_n}{L^2(]0,T[,H^{-1})} \sqrt{t_1-t_2}.
\end{multline*}
Since the sequence $(\partial_t u_n)_\nnnr$ is bounded in $L^2(]0,T[,H^{-1}(\O))$, one deduces\\
$\norm{u_n(t)-u_n(s)}{H^{-1}} \to 0$, as $s \to t$, uniformly with respect to $\nnnr$ (and for all $t \in [0,T]$).

In order to prove the first item, one uses Estimate \eqref{est-ld}. 
It gives that 
the sequence $(u_n(t))_\nnnr$ is bounded in $L^2(\O)$ for all $t \in [0,T]$ and then is relatively compact in $H^{-1}(\O)$ for all $t \in [0,T]$.

It is then possible to apply Ascoli's Theorem and obtain as it is said before 
$u(0)=u_0$.
This proves that $u$ is solution of \eqref{stefan}.

\medskip

Finally, we will prove now that $\vphi(u_n) \to \vphi(u)$  in $L^2(]0,T[,H^1_0(\O))$ as $\nti$.
Let $\Phi$ be the primitive of $\vphi$ defined by $\Phi(s) = \int_0^s \vphi(t)\dt$.
The convexity of $\Phi$, the fact that $\vphi(u) \in L^2(]0,T[, L^2(\O))$ and $u_n \to u$ weakly in $L^2(]0,T[, L^2(\O))$ gives
\[
\int_0^T \int_\O \Phi(u) \dx \dt \le \liminf_\nti \int_0^T \int_\O \Phi(u_n) \dx \dt.
\]
Since $\vphi(u_n) \to \vphi(u)$ weakly in $L^2(]0,T[,H^1_0(\O))$, as $\nti$, one has also
\[
\int_0^T \int_\O \nabla u \cdot \nabla u \dx \dt \le \liminf_\nti \int_0^T \int_\O \nabla u_n \cdot \nabla u_n \dx \dt.
\]
Taking $v = \vphi(u_n)$ in  \eqref{eq-pf} and  $v = \vphi(u)$ in \eqref{stefan} considered for any $t\in[0,T]$, one proves 
\begin{multline*}
 \limsup_{\nti} (\int_0^T \int_\O \Phi(u_n) \dx \dt +   \int_0^T \int_\O \nabla u_n \cdot \nabla u_n \dx \dt)
 \\
 \le\int_0^T \int_\O \Phi(u) \dx \dt +   \int_0^T \int_\O \nabla u \cdot \nabla u \dx \dt.
\end{multline*}
Then, we get
\[
 \lim_{\nti} \norm{\vphi(u_n)}{L^2(]0,T[,H^1_0(\O))} =  \norm{\vphi(u)}{L^2(]0,T[,H^1_0(\O))}
\]
and this gives the convergence, as $\nti$, of $\vphi(u_n)$ to $\vphi(u)$  in $L^2(]0,T[,H^1_0(\O))$.

\section{The case of irregular data}\label{sec:irregrhs}

\subsection{Weak formulation and convergence result}

We now consider the equation 
$\partial_t u  - \Delta(\vphi(u))=f$,  with a homogeneous Dirichlet boundary condition on $\vphi(u)$, $f\in L^1(]0,T[,L^1(\O))$ and  initial datum $u_0\in L^1(\O)$. 
For any $p \ge 1$, we let $p' = \frac p {p-1}$ if $p>1$, $p'=\infty$ if $p=1$. 
We define the space $W^{-1,1}_\star(\Omega)$ as the dual space of $W^{1,\infty}_0(\O)$ (see \cite[Remark A.3]{gal2012conv}). We recall that, for $1 \le p<+\infty$, the space $L^p(Q_T)$ can be identified with the space  $L^p(]0,T[,L^p(\O))$, which is not true for $p=+\infty$
(see \cite{droniou-ivv}, Section 1.8.1).

Since $\ldo$ is identified with $\ldo'$, one has for any $p\in ]1,2]$, 
\[
 W^{1,\infty}_0(\O)\subset W^{1,p'}_0(\Omega)\subset \ldo = \ldo' \subset W^{-1,p}(\Omega)\subset W^{-1,1}_\star(\Omega).
\]
We again use the truncation function $T_k:\R\to\R$ defined by \eqref{eq:deftrunc}, that is
\[
 \forall k>0,\ \forall s\in\R,\ T_k(s) = \min(k,|s|){\rm sign}(s).
\]

The weak sense that we consider is given by

\begin{subequations}
\begin{align}
	&\left\{\begin{array}{l}
	u \in L^{\infty}(]0,T[,L^{1}(\Omega)), \, u \in C([0,T], W^{-1,1}_\star(\Omega))), 
	\\
	\hbox{ for any }r <\frac d {d-1},\ u \in L^{2}(]0,T[,L^{r}(\Omega)),\\
	\partial_t u  \in L^1(]0,T[, W^{-1,1}_\star(\Omega)),
	\\
	 \hbox{ for any }p<\frac {d+2}{d+1},
	\vphi(u) \in L^p(]0,T[,W^{1,p}_0(\Omega)),\\
	\hbox{ for any }k>0,\ T_k(\vphi(u)) \in L^2(]0,T[,H^{1}_0(\Omega)),\\
	\dsp \int_0^T \left\langle \partial_t u(s) ,\psi(s)\right\rangle_{W^{-1,1}_\star,W^{1,\infty}_0} \ds + \int_0^T \int_\Omega \grad \vphi(u(x,s))  \cdot \nabla \psi(x,s) \dx \ds  \\
	 \hfill =\dsp \int_0^T \int_\Omega   f(x,s) \psi(s)  \dx \ds , \qquad \forall \psi \in L^{\infty}(]0,T[,W^{1,\infty}_0(\Omega)), 
	\end{array}\right. \label{eq:pfweakirr}\\
&	u(\cdot,0 )= u_0.\label{eq:pfweakirrz}
	\end{align}
\label{eq-pf-weak-irr}
\end{subequations}
 Notice that we include in the above weak sense the regularity property $T_k(\vphi(u))\in L^2(]0,T[,H^{1}_0(\Omega))$. The question of the uniqueness of $u$ solution to \eqref{eq-pf-weak-irr} remains open (see \cite{ben1997unicite} for an example of a uniqueness property in the elliptic case, involving the regularity of the truncations).
 
In order to prove the existence of a solution to Problem \eqref{eq-pf-weak-irr}, we complete the hypotheses on $\vphi$ by additional ones which are not needed in  Section \eqref{sec:regrhs}:
 \begin{subequations}
\begin{align}
 & \vphi : \mathbb{R} \to \mathbb{R} \mbox{ is non-decreasing and $L_\vphi$-Lipschitz continuous,}\label{hyp:lipphi}\\
 & \vphi(0) = 0\label{hyp:phizero}\mbox{ and }\\
&\mbox{ there exist $Z_0,Z_1>0$ such that $|\vphi(s)| \ge Z_1|s| - Z_0$ for any $s \in \mathbb{R}$.}\label{hyp:surlin}
\end{align}
\label{eq:hypzeta}
\end{subequations}
Hypothesis \eqref{hyp:lipphi} is already used in Section 2. Hypothesis \eqref{hyp:phizero}, which does not reduce the generality, is only done for simplifying some computations in the proofs below. Hypothesis \eqref{hyp:surlin} is used for getting some estimates allowing the convergence analysis done in this irregular data case.

In this section \ref{sec:irregrhs}, we consider a sequence $(f^{(n)},u_0^{(n)})_{\nnn}$ such that:

 $f^{(n)}\in L^2(]0,T[,\ldo)$, $u_0^{(n)}\in L^2(\O)$, $f^{(n)}\to f$ in $L^1(\O \times ]0,T[)$ and $u_0^{(n)}\to u_0$ in $L^1(\O)$ as $\nti$.

We then consider, for all $\nnn$, a function $u^{(n)}\in L^2(]0,T[,\ldo)$ such that \eqref{stefan} holds with $u_0 = u_0^{(n)}$ and $f = f^{(n)}$, which is expressed by
\begin{equation}
	\left\{\begin{array}{l}
	u^{(n)} \in L^{\infty}(]0,T[,L^2(\Omega)), \partial_t u^{(n)}  \in L^2(]0,T[, H^{-1}(\Omega)), u^{(n)} \in C([0,T], H^{-1}(\Omega)),\\
	\vphi(u^{(n)}) \in L^{2}(]0,T[,H^1_0(\Omega)), \\
	\dsp \int_0^T \left\langle \partial_t u^{(n)}(s) , v(s)\right\rangle_{H^{-1},H^{1}_0 } \ds + \int_0^T \int_\Omega \grad \vphi(u^{(n)}(x,s))  \cdot \nabla v(x,s) \dx \ds  \\
	 \hfill =\dsp \int_0^T \int_\Omega   f^{(n)}(x,s) v(x,s)  \dx \ds , \qquad \forall v \in L^2(]0,T[, H^{1}_0(\Omega)), \\
	u^{(n)}(\cdot,0 ) = u_0^{(n)}.
	\end{array}\right. \label{stefann}
\end{equation}
The existence of such a function $u^{(n)}$ is proved in Section \eqref{sec:regrhs}.
We then have the following convergence result for $d<4$ (this limitation  is due to the fact that we need the compactness of $L^r(\O)$ in $H^{-1}(\O)$, which  holds for $r>\frac {2d}{d+2}$, see Lemma \ref{lem:estimwmun}, although the approximate solution is bounded for $r< \frac d {d-1}$, as stated in Lemma  \ref{lem:estimhun} and $\frac {2d}{d+2} < \frac d {d-1}$ only for $d<4$).

\begin{theorem}[Convergence of the regularisation method]\label{thm:irreg}
Let $\O$ be an open bounded subset of $\R^d$, with $d=2$ or $3$ and let $T>0$. Let $\vphi$ be given such that \eqref{eq:hypzeta} holds. Let  $f\in L^1(]0,T[,L^1(\O))$ and  $u_0\in L^1(\O)$ be given, and let $(f^{(n)},u_0^{(n)})_{\nnn}$ be such that $f^{(n)}\in L^2(]0,T[,\ldo)$, $u_0^{(n)}\in L^2(\O)$ and $f^{(n)}\to f$ in $L^1(\O \times ]0,T[)$ and $u_0^{(n)}\to u_0$ in $L^1(\O)$ as $\nti$. 
 Let $(u^{(n)})_\nnn$ be a solution to \eqref{stefann} for all $\nnn$. 
 
 Then there exists $u$, solution of \eqref{eq-pf-weak-irr}, such that, as $\nti$ up to some subsequence,
 \begin{itemize}
  \item $u^{(n)}$ converges to $u$ in $C([0,T],W^{-1,1}_\star(\O))$, $L^2(]0,T[,H^{-1}(\O))$ and weakly in $L^2(]0,T[,L^r(\O))$ for any $r\in [1,\frac d {d-1}[$,
  \item $\vphi(u^{(n)})$ converges to $\vphi(u)$ in $L^p(]0,T[,L^p(\O))$ and weakly in $L^p(]0,T[,W^{1,p}_0(\O))$ for any $p\in [1,\frac {d+2} {d+1}[$,
  \item for all $k>0$, $T_k(\vphi(u^{(n)}))$ converges to $T_k(\vphi(u))$ in $L^2(]0,T[,L^2(\O))$.
 \end{itemize}

\end{theorem}

The remaining of this section is devoted to the proof of Theorem \ref{thm:irreg}.

\subsection{Estimates on the solution to the regularised problem}

 Since all the estimates of Section \eqref{sec:regrhs} are functions of the $L^2$-norm of the data $f$ and $u_0$, specific estimates with respect to the $L^1$-norm of the data $f$ and $u_0$ have to be given. 

In the whole section, we denote by $Q_T = \Omega\times]0,T[$.

We denote by $\ctel{cteun}$ a bound of $\Vert f^{(n)}\Vert_1$ and $\Vert u_0^{(n)}\Vert_1$ for all $\nnn$.

Let $\psi~:~\R\to (-1,1)$, and $\beta~:~\R\to \R$ be defined by
\begin{equation} \label{eq:def:psi }
  \begin{array}{ll} \displaystyle
   \dsp \forall s\in\R,\ \psi (s)= \frac {\ln(1+|s|)} {1+\ln(1+|s|)}{\rm sign}(s),\\
	\dsp \forall s\in\R,\ 	\beta(s) = \int_0^s \sqrt{\psi'(t)}\dt,
  \end{array}
\end{equation}
where ${\rm sign}(s)=1$ if $s\ge 0$ and $-1$ if $s<0$. Note that the reciprocal of $\psi$ is used in \cite{eym2021conv} for the study of the convergence of a regularised finite element scheme to the solution of an elliptic problem with $L^1$ data.
 The inequality 
$\frac 1 {\psi'(s)} \le \frac{4(1+s)^{1+\tau}} {\tau^2}$, for all $s\ge 0$ and $\tau\in ]0,2[$, implies that
\begin{equation}\label{eq:propbeta}
 \forall q\in ]0,\frac 1 2[,\ \forall s\in\R,\  (1-2q) ((1+|s|)^q - 1) \le |\beta(s)|.
\end{equation}

Let us now give estimates enabling some compactness on the sequence $(u^{(n)})_{\nnn}$.

\begin{lemma}\label{lem:estimhun} Let $\O$ be an open bounded subset of $\R^d$, with $d=2$ or $3$ and let $T>0$. Let $\vphi$ be given such that \eqref{eq:hypzeta} holds. Let  $g \in L^2(]0,T[ , L^2(\Omega))$, $v_0 \in L^2(\O)$ and 
$v$ be a solution to Problem \eqref{stefan} with $f = g$ and $u_0 = v_0$.  Let $\psi$  and $\beta$ be defined by \eqref{eq:def:psi }. Then
there exists $\ctel{ctedeux}$, only depending on $\Omega$, $T$, $Z_0$ and $Z_1$ and increasingly depending on $\Vert g\Vert_1 +\Vert v_0\Vert_1$ such that the following inequalities hold :
 \begin{equation}\label{estiml1}
  \Vert \vphi(v)\Vert_{L^\infty(]0,T[,L^1(\O))} \le \cter{ctedeux} 
  \hbox{ and }\Vert v\Vert_{L^\infty(]0,T[,L^1(\O))}\le \cter{ctedeux},
  \end{equation}
 \begin{equation}\label{estimhunbeta}
  \Vert \nabla \beta(\vphi(v))\Vert_{L^2(]0,T[,\ldo)} \le \cter{ctedeux},
  \end{equation}
  there exists a function $\ctel{ctedeuxr}(r)$  only depending on $d$, $\Omega$, $T$, $Z_0$ and $Z_1$ and increasingly depending on $\Vert g\Vert_1 +\Vert v_0\Vert_1$ such that
 \begin{equation}\label{estiml2}
  \hbox{for any }r<\frac d {d-1},\ \norm{v}{L^2(]0,T[,L^{r}(\O))}\le \cter{ctedeuxr}(r),
  \end{equation}
  and
  \begin{equation}\label{estiml2b}
  \norm{v}{L^2(]0,T[,H^{-1}(\O))}\le  \cter{ctedeux}.
  \end{equation}

\end{lemma}
Note that in Lemma \ref{lem:estimhun}, the estimate on $\norm{v}{L^2(]0,T[,H^{-1}(\O))}$ is obtained with the estimate on 
$\norm{v}{L^2(]0,T[,L^{r}(\O))}$ taking $r \ge 2d/(d+2)$ which is possible thanks to $d \le 4$. 
The fact that $d<4$ will be needed in the sequel since we will need the compactness of $L^r$ in $H^{-1}$ which is true for $r>2d/(d+2)$.

\begin{proof}
In this proof, we denote by $C_i$, for $i\in \mathbb{N}$, various nonnegative real values or functions which only depend on $d$, $\O$, $T$, $Z_0$, $Z_1$, $L_\vphi$ and increasingly depending on 

$\Vert g\Vert_{L^1(Q_T)}+\Vert v_0\Vert_{L^1(\O)}$.
 We define the function $A(s) = \int_0^s\psi(\vphi(a)){\rm d}a \ge 0$, which satisfies, using \eqref{hyp:lipphi}-\eqref{hyp:phizero}
 \[
  \forall s\in\R, 0\le A(s)\le |s|,
 \]
 We notice that, letting $T_0 = \psi^{-1}(\frac 1 2)$, we have $|\psi(t)|\ge \frac 1 2$ for any $|t|\ge T_0$. This implies, using  \eqref{hyp:surlin}, that
 \[
  |\psi(\vphi(a))|\ge \frac 1 2\hbox{ for any }|a|\ge \frac {Z_0+ T_0} {Z_1},
 \]
which yields
 \begin{equation}\label{eq:propA1}
  \forall s\in\R,  A(s) \ge \frac 1 2 \max(|s| -\frac {Z_0+ T_0} {Z_1},0)\ge \frac 1 2 |s| - \frac {Z_0+ T_0} {2 Z_1}.
 \end{equation}
 Using \eqref{hyp:lipphi}-\eqref{hyp:phizero}, we also have
 \begin{equation}\label{eq:propA2}
  \forall s\in\R,   (2 A(s)+\frac {Z_0+ T_0} {Z_1})L_\vphi \ge  |s| L_\vphi \ge   |\vphi(s)|.
 \end{equation}
Letting $\psi(\vphi(v))$ as test function in \eqref{stefan}, we get the following inequality, for any $t\in [0,T]$:
\[
\Vert A(v(t))\Vert_{L^1(\O)}+ \Vert \nabla \beta(\vphi(v))\Vert_{L^2(]0,t[,\ldo)}^2 \le \Vert g\Vert_{L^1(Q_T)}+\Vert v_0\Vert_{L^1(\O)}, 
\]
since $\Vert A(v_0)\Vert_{L^1(\O)}\le \Vert v_0\Vert_{L^1(\O)}$. Hence we obtain  \eqref{estimhunbeta} letting $t=T$.
Since the above inequality holds for any $t\in[0,T]$, using \eqref{eq:propA1}-\eqref{eq:propA2} provides  \eqref{estiml1}.

Let use now consider $d> 2$. We get, from the same inequality and using a Sobolev inequality,
\[
\Vert \beta (\vphi(v))\Vert_{L^2(]0,T[, L^{\frac {2d}{d-2}}(\O))}^2\le \ctel{ctetrois},
\]
which yields, using \eqref{eq:propbeta} and accounting for \eqref{hyp:surlin},
\[
\int_0^T \Big(\int_\O |v|^{q}\dx\Big)^{\frac {d-2}{d}} \dt \le \ctel{ctequatre}(q)\hbox{ for any }q\in ]0,\frac {d}{d-2}[.
\]
We have the following H\"older inequality, for any $1<r<q<\frac {d}{d-2}$:
\[
 \Big(\int_\O |v|^{r} \dx\Big)^{2/r}\le 
 \Big(\int_\O |v| \dx\Big)^{\frac {2(q-r)}{r(q-1)}}\Big(\int_\O |v|^{q} \dx\Big)^{\frac {2(r-1)}{r(q-1)}}.
\]
Since we need the inequality
\[
 \frac {2(r-1)}{r(q-1)}\le \frac {d-2} d,
\]
the smallest possible value for $q$ is given by the relation
\begin{equation}\label{eq:defq}
 \frac {2(r-1)}{r(q-1)} = \frac {d-2} d,\hbox{ which means }q = 1 + \frac {2(r-1)d} {r(d-2)}.
\end{equation}
Then the relation $q<\frac {d}{d-2}$ implies $r<\frac {d}{d-1}$. For this choice of $q$, we obtain, using \eqref{estiml1},  
\[
 \int_0^T\Big(\int_\O |v|^{r} \dx\Big)^{2/r}\dt \le \ctel{ctekljdsgh}\int_0^T\Big(\int_\O |v|^{q} \dx\Big)^{\frac {d-2} d}\dt \le \cter{ctekljdsgh}\cter{ctequatre}(q).
\]
Since our aim is to use the compact embedding of $L^r(\O)$ in $H^{-1}(\O)$, it suffices to select $r$ such that
\[
 \frac 1 {1 - \frac {d-2}{2d}} = \frac {2d}{d+2} < r <\frac d {d-1},
\]
(which is possible for $d=3$ but impossible for $d\ge 4$). Then for any such $r$, defining $q$ by \eqref{eq:defq}, all the needed inequalities on $r$, $q$ and $d$ hold. We then obtain \eqref{estiml2}.

For $d=2$, we can select $q = 3$, for obtaining the same conclusion for any $ \frac {2d}{d+2} < r <\frac d {d-1} $.
\end{proof}

\begin{lemma}\label{lem:estimlp} Let $\O$ be an open bounded subset of $\R^d$, with $d\in\N^\star$ and let $T>0$. Let $\psi$  and $\beta$ be defined by \eqref{eq:def:psi }.
Let  $v\in L^\infty(]0,T[,L^1(\O))$ such that $\beta(v)\in L^2(]0,T[,H^1_0(\O))$.
Then, for any $ 1 \le p < \frac{d+2}{d+1}$, there exists $\ctel{ctehunbetav}$, only depending on $\O$, $T$, $p$ and $d$ and increasingly depending on $\norm{v}{L^\infty(]0,T[,L^1(\O))}+\Vert \nabla \beta(v)\Vert_{L^2( ]0,T[,\ldo)}$  such that
 \begin{equation}\label{lin:eq:estimnablau}
 \Vert v\Vert_{L^p(]0,T[,W^{1,p}_0(\O))} \dx
\le \cter{ctehunbetav}.
\end{equation}
\end{lemma}
\begin{proof}
In this proof, we denote by $C_i$, for $i\in \mathbb{N}$, various nonnegative real values or functions which only depend on $\O$, $T$ and increasingly depending on  $\norm{v}{L^\infty(]0,T[,L^1(\O))}+\Vert \nabla \beta(v)\Vert_{L^2(]0,T[,\ldo)}$.
Let  $\ctel{lin:cste:l2}>0$ be such that
\begin{equation}\label{estimlinflunbetav}
\norm{v}{L^\infty(]0,T[,L^1(\O))}\le \cter{lin:cste:l2},
\end{equation} 
and 
 \begin{equation}\label{estimhunbetav}
  \Vert \nabla \beta(v)\Vert_{L^2(]0,T[,\ldo)}^2 = \int_0^T\int_\O \psi'(v)|\nabla v|^2\dx\dt \le \cter{lin:cste:l2}.
  \end{equation}

Let $ 1 \le p < \frac{d+2}{d+1}$. 
Using Hölder's inequality with conjugate exponents $ \frac{2}{p}>1 $ and $\frac{2}{2-p} $ and owing to \eqref{estimhunbetav}, we obtain
\begin{multline}\label{eq:pp0}
\int_{Q_T} |\nabla v|^p \dx\dt
=
\int_{Q_T} |\nabla v|^p 
\left(
\frac
{\psi'(v)}
{\psi'(v)}
\right)^{p/2} \dx\dt
\\
\leq
\left(\int_{Q_T}\psi'(v)|\nabla v|^2 \dx\dt\right)^{p/2}
\left(\int_{Q_T}\frac 1 {(\psi'(v))^{p/(2-p)}}\dx\dt\right)^{(2-p)/2} \\
\le \left(\cter{lin:cste:l2}\right)^{p/2}
\left(\int_{Q_T} \frac 1 {(\psi'(v))^{p/(2-p)}}\dx\dt\right)^{(2-p)/2}.
\end{multline}
Our aim is now to bound the $L^{p/(2-p)}$ norm of  $1/\psi' ( v)$, using Sobolev inequalities. 
From the properties of $\psi$, we have, for any $\theta\in ]1,3[$ and $s\ge 0$,
\begin{equation}\label{eq:bounduspsiprime}
\frac 1 {\psi'(s)} \le \frac{4(1+s)^{\theta}} {(\theta-1)^2}.
\end{equation}
We thus obtain that there exist $\ctel{lin:cte:alphataur}^{(p,\theta)}$ and $\ctel{lin:cte:betataur}^{(p,\theta)}$ such that
\begin{equation}\label{eq:pp1}
\left(\int_{Q_T}
\frac 1 {(\psi'(v))^{p/(2-p)}}\dx\dt\right)^{(2-p)/2}
\leq
\cter{lin:cte:alphataur}^{(p,\theta)}
\left(\int_{Q_T}
|v|^{\frac {\theta p} {2-p}}\dx\dt\right)^{(2-p)/2}
+
\cter{lin:cte:betataur}^{(p,\theta)}.
\end{equation}

Let us now conclude the proof, following the method used in \cite{boc1989nonlin} (written in a discrete way in \cite{gal2012conv}).
We recall the Sobolev inequality, for $p<d$,
\begin{equation}\label{eq:sobolev}
\| v \|_{L^{p^\star}(\Omega)}^p \le C_{{\rm sob}}^{(p)} \| \nabla v \|_{L^p(\Omega)}^p,\mbox{ for any }v \in W^{1,p}_0(\Omega),
\end{equation}
with $p^\star = \frac {d p }{d-p}$.
We recall the interpolation property, for any $1<r<p^\star$:
\begin{equation}\label{eq:intlplq}
 \forall w\in L^{p^\star}(\O),\ \Vert w\Vert_{L^r(\O)} \le \Vert w\Vert_{L^{p^\star}(\O)}^{\zeta} \Vert w\Vert_{L^1(\O)}^{1-\zeta}\hbox{ with }\zeta = \frac {1-1/r}{1-1/{p^\star}}.
\end{equation}
We choose $\theta>1$ such that $r = \frac {\theta p}{2-p}$ verifies $1<r<p^\star$, which means
\begin{equation}\label{eq:ppth1}
 1<\theta < \theta_1 \hbox{ with }\theta_1 := \frac {d(2-p)} {d-p}\in ]\frac {d^2} {d^2- 2}, \frac {d} {d-1}].
\end{equation}
For such a choice for $\theta$, we get, from the hypothesis \eqref{estimlinflunbetav} and using \eqref{eq:intlplq},
\begin{multline}
 \Big(\int_{Q_T}
|v|^{\frac {\theta p} {2-p}}\dx\dt\Big)^{(2-p)/2} = \Big(\int_0^T \Vert v(\cdot,t)\Vert_{L^{r}(\O)}^{r}\dx\dt\Big)^{(2-p)/2} \\
\le \ctel{ctedix}\Big(\int_0^T\Vert v(\cdot,t)\Vert_{L^{p^\star}(\O)}^{\zeta r}\dx\dt\Big)^{(2-p)/2}.
\label{eq:pp3}
\end{multline}
In order to apply the H\"older inequality
\begin{equation}\label{eq:pp4}
 \Big(\int_0^T\Vert v(\cdot,t)\Vert_{L^{p^\star}(\O)}^{\zeta r}\dt\Big)^{(2-p)/2}\le \ctel{cteonze}\Big(\int_0^T\Vert v(\cdot,t)\Vert_{L^{p^\star}(\O)}^{p}\dt\Big)^{\beta}\hbox{ with }\beta :=\frac{(2-p)r\zeta}{2p},
\end{equation}
the condition $\zeta r < p$ must be satisfied. Using the above defined values for $\zeta$ and $r$, this inequality holds under the sufficient condition
\begin{equation}\label{eq:ppth2}
 1 < \theta < \theta_2\hbox{ with }\theta_2 := \frac {d+1} d (2-p)\in ]1, \frac {d+1} d].
\end{equation}
It suffices to define
\begin{equation}\label{eq:ppth3}
 \theta = \frac 1 2\min \Big( 1+\theta_1,1+\theta_2)
\end{equation}
for simultaneously satisfying \eqref{eq:ppth1} and \eqref{eq:ppth2}. For such a value for $\theta$, the exponent $\beta$ defined in \eqref{eq:pp4} satisfies
\begin{equation}\label{eq:pp9}
 \beta <\frac{(2-p)p}{2p} < 1.
\end{equation}
From \eqref{eq:sobolev}, \eqref{eq:pp0}, \eqref{eq:pp1}, \eqref{eq:pp3} and \eqref{eq:pp4}, we deduce
\[
\int_{Q_T} |\nabla v|^p \dx\dt\le  \left(\cter{lin:cste:l2}\right)^{p/2}\Big(
\cter{lin:cte:alphataur}^{(p,\theta)}
\cter{ctedix}\cter{cteonze}\Big(C_{{\rm sob}}^{(p)} \int_{Q_T}|\nabla v|^{p}\dx\dt\Big)^{\beta}
+
\cter{lin:cte:betataur}^{(p,\theta)}\Big),
\]
which implies \eqref{lin:eq:estimnablau} accounting for \eqref{eq:pp9}.
\end{proof}
\begin{lemma}\label{lem:estimtk} Under the assumptions of Theorem \ref{thm:irreg}, there exists $\ctel{csttk}$, which only depends on $d$, $\O$, $T$, $\vphi$, $\Vert f\Vert_1$ and $\Vert u_0\Vert_1$, such that, for any $k>0$, 
 \begin{equation}\label{lin:eq:estimmupsi}
 \norm{\nabla T_k(\vphi(u^{(n)}))}{L^2(]0,T[,\ldo)}
\le k(k+1)\cter{csttk}.
\end{equation}
\end{lemma}
\begin{proof}
 We let $ T_k(\vphi(u^{(n)}))$ as test function in \eqref{stefan}. We get $A_1 + A_2 = A_3$, with
 \[
  A_1 = \int_0^T \left\langle \partial_t u^{(n)}(s) ,  T_k(\vphi(u^{(n)})(s))\right\rangle_{H^{-1},H^{1}_0} \ds,
 \]
\[
 A_2 =   \int_0^T \int_\Omega \grad \vphi(u^{(n)}(x,s))  \cdot \nabla T_k(\vphi(u^{(n)}(x,s)) \dx \ds = \norm{\nabla T_k(\vphi(u^{(n)}))}{L^2(]0,T[,\ldo)}^2,
\]
and
\[
 A_3 = \int_0^T \int_\Omega   f^{(n)}(x,s) T_k(\vphi(u^{(n)}(x,s))  \dx \ds.
\]
We have
\[
 |A_3|\le k \cter{cteun}.
\]
Letting $\Theta_k(s) = \int_0^s T_k(\vphi(t))\dt$ for all $s\in\R$, we have
\[
  A_1 = \int_\O  \Theta_k(u^{(n)}(x,T))\dx -  \int_\O  \Theta_k(u_0^{(n)}(x))\dx. 
 \]
 Note that, since $0\le \Theta_k(s) \le |\int_0^s T_k(L_\vphi t)\dt|$ for all $s\in\R$, we have
 \[
 \forall s\in\R,\  0 \le \Theta_k(s) \le  k (L_\vphi\frac k 2 + |s|).
 \]
 This implies that 
 \[
  A_1 \ge -k  \Big(L_\vphi\frac k 2 |\O| + \norm{u_0^{(n)}}{L^1(\O)}\Big).
 \]
 Gathering these relations leads to \eqref{lin:eq:estimmupsi}.
\end{proof}

\begin{lemma}[Compactness properties of $(u^{(n)})_{\nnn}$]\label{lem:estimwmun} Under the assumptions of Theorem \ref{thm:irreg}, there exists a subsequence of $(u^{(n)})_{\nnn}$, again denoted 
$(u^{(n)})_{\nnn}$ and functions $u$ and $v$, such that 
\begin{itemize}
 \item $u^{(n)}$ converges to $u$ in $L^{2}(]0,T[,H^{-1}(\O))$.
 \item $u^{(n)}$ weakly converges to $u$ in $L^{2}(]0,T[,L^r(\O))$ for any $r\in [1,\frac d {d-1}[$.
 \item $\partial_t u_n$ weakly converges to  $\partial_t u$ in $L^1(]0,T[,W^{-1,1}_\star(\O))$
 \item $\vphi(u^{(n)})$ weakly converges to $v$ in $L^p(]0,T[,W^{1,p}_0(\O))$ for any $p\in[1,\frac {d+2}{d+1}[$.
\end{itemize}

\end{lemma}
\begin{proof}
The proof is done by following a series of compactness steps.
\begin{enumerate}
 \item[Step 1.] Using \eqref{estiml2} in Lemma \ref{lem:estimhun}, we extract a subsequence such that, for a given $r\in [1,\frac d {d-1}[$, $u^{(n)}$ weakly converges to some $u\in L^2(]0,T[,L^{r}(\O))$.
 \item[Step 2.] By uniqueness of the limit, we then obtain that $u^{(n)}$ weakly converges to $u$ for any $r\in [1,\frac d {d-1}[$.
 \item[Step 3.] Using Lemma \ref{lem:estimlp} thanks to \eqref{estiml1} and \eqref{estimhunbeta}  in Lemma \ref{lem:estimhun}, we get that there exists $v\in L^p(]0,T[,W^{1,p}_0(\O))$ and a subsequence of the preceding one such that, for a given $p\in[1,\frac {d+2}{d+1}[$, $\vphi(u^{(n)})$ weakly converges to $v$ in $L^p(]0,T[,W^{1,p}_0(\O))$.
 \item[Step 4.] By uniqueness of the limit, we then obtain that for all $p\in[1,\frac {d+2}{d+1}[$, $\vphi(u^{(n)})$ weakly converges to $v$ in $L^p(]0,T[,W^{1,p}_0(\O))$.
 \item[Step 5.] Then the linear form 
 \begin{multline*}
  D^{(n)}\in L^1(]0,T[,W^{-1,1}_\star(\O)),\ \langle D^{(n)},w\rangle =  -\int_0^T \int_\Omega \grad \vphi(u^{(n)}(x,s))  \cdot \nabla w(x,s) \dx \ds\\ +\int_0^T \int_\Omega   f^{(n)}(x,s) w(x,s)  \dx \ds
 \end{multline*}
weakly converges to
 \begin{multline*}
  D\in L^1(]0,T[,W^{-1,1}_\star(\O)),\ \langle D,w\rangle =  -\int_0^T \int_\Omega \grad v(x,s)  \cdot \nabla w(x,s) \dx \ds\\  +\int_0^T \int_\Omega   f(x,s) w(x,s)  \dx \ds,
 \end{multline*}
 by the weak convergence of $\grad \vphi(u^{(n)})$ to $\grad v$ in $L^p(]0,T[,L^p(\O))$ and the convergence in $L^1$ of $f^{(n)}$ to $f$.
 \item[Step 6.] Since $D_n=\partial_t u_n$, we obtain that $\partial_t u_n$ weakly converges to $D$. Since, for any $\psi\in C^\infty_c(Q_T)$,
\[
  \int_0^T \left\langle \partial_t u^{(n)}(s) ,  \psi(s)\right\rangle_{W^{-1,1}_\star,W^{1,\infty}_0} \ds = - 
  \int_0^T\int_\O u^{(n)}(x,s)\partial_t \psi(x,s)\dx\ds,
 \]
 we get that, letting $\nti$
 \[
  \langle D,\psi\rangle = - \int_0^T\int_\O u(x,s)\partial_t \psi(x,s)\dx\ds,
 \]
 which proves that $D = \partial_t u\in L^1(]0,T[,W^{-1,1}_\star(\O))$.
\item[Step 7.] Since, for a given $r\in ]\frac {2d}{d+2},\frac d {d-1}[$, $L^{r}(\O)$ is compactly embedded in $H^{-1}(\O)$, since $u^{(n)}$ weakly converges to $u$ in $L^{2}(]0,T[,L^{r}(\O))$ and since $\partial_t u^{(n)}$ weakly converges in $L^1(]0,T[,W^{-1,1}_\star(\O))$ (it is therefore bounded), we deduce, by \cite[Theorem 4.42]{edp-gh} with $B=H^{-1}(\O)$, $Y = W^{-1,1}_\star(\O)$ and $X=L^{r}(\O)$ that $u^{(n)}$ converges to $u$ in $L^{2}(]0,T[,H^{-1}(\O))$.
\end{enumerate}
\end{proof}

It remains now to prove that $v = \vphi(u)$ and that $u(0) = u_0$. This is the aim of the next sections.

\subsection{Minty trick}

\begin{lemma}\label{lem:mintytk} Under the assumptions of Theorem \ref{thm:irreg}, let $u$, $v$, $(u^{(n)})_{\nnn}$ be given by Lemma \ref{lem:estimwmun}. Then
\begin{itemize}
\item $v = \vphi(u)$.
 \item $T_k(\vphi(u^{(n)}))$ converges to $T_k(\vphi(u))$ in $L^2(]0,T[,L^2(\O))$
 \item  $\vphi(u^{(n)})$ converges to $\vphi(u)$ in $L^p(]0,T[,L^p(\O))$ for all $p\in[1,\frac {d+2}{d+1}[$.
\end{itemize}
\end{lemma}
\begin{remark}
 The convergence of $\vphi(u^{(n)})$ to $\vphi(u)$ in $L^p(]0,T[,L^p(\O))$ for all $p\in[1,\frac {d+2}{d+1}[$ is also a consequence of the convergence of $\beta(\vphi(u^{(n)}))$ to $\beta(\vphi(u))$ in $L^2(]0,T[,$ $L^2(\O))$, which is a consequence of Lemma \ref{lem:mintyneg} applied to the function $\beta\circ\vphi$.
\end{remark}

\begin{proof}
We reason as in Section \ref{sec:minty}. Let $k>0$.
Since $T_k(\vphi(u^{(n)}))$ is bounded in $L^{2}(]0,T[,H^1_0(\Omega))$, and $T_k\circ\vphi$ is nondecreasing, applying Minty's Trick Lemma \ref{lem:mintyneg} with $u^{(n)}$ instead of $u_n$ and $T_k\circ\vphi$ instead of $\vphi$, we get that $T_k(\vphi(u^{(n)}))$ converges to $T_k(\vphi(u))$ in $L^{2}(]0,T[,\ldo)$.
Since this holds for all $k>0$, by Lemma \ref{lem:estimwmun} and Lemma \ref{lem:cvtkstrong}, we then deduce for all $p \in [1,(d+2)/(d+1)[$ the convergence of $\vphi(u^{(n)})$ to $\vphi(u)$ in $L^p(]0,T[,L^p(\O))$.
\end{proof}

\subsection{Conclusion of the proof of Theorem \ref{thm:irreg}}

We apply Lemma \ref{lem:estimwmun}. Let us prove that $u$ is solution to \eqref{eq:pfweakirr}. We let $\psi \in L^{\infty}(]0,T[,W^{1,\infty}_0(\O))$ as test function in \eqref{stefann}. We get
\begin{equation*}
	\left\{\begin{array}{l}
	\dsp \int_0^T \left\langle \partial_t u^{(n)}(s) , \psi(s)\right\rangle \ds + \int_0^T \int_\Omega \grad \vphi(u^{(n)}(x,s))  \cdot \nabla \psi(x,s) \dx \ds  \\
	 \hfill =\dsp \int_0^T \int_\Omega   f^{(n)}(x,s) \psi(x,s)  \dx \ds.
	\end{array}\right. 
\end{equation*}
We have, by weak convergence of $\partial_t u^{(n)}(s)$ in $L^{1}(]0,T[,W^{-1,1}_\star(\O))$, that
\[
 \lim_{\nti}\int_0^T \left\langle \partial_t u^{(n)}(s) , \psi(s)\right\rangle \ds = \int_0^T \left\langle \partial_t u(s) , \psi(s)\right\rangle \ds.
 \]
From Lemma \ref{lem:mintytk}, we have
\[
 \lim_{\nti}\int_0^T \int_\Omega \grad \vphi(u^{(n)}(x,s))  \cdot \nabla \psi(x,s) \dx \ds = \int_0^T \int_\Omega \grad \vphi(u(x,s))  \cdot \nabla \psi(x,s) \dx \ds,
\] 
and  we get that
\[
 \lim_{\nti} \int_0^T \int_\Omega   f^{(n)}(x,s) \psi(x,s)  \dx \ds =  \int_0^T \int_\Omega   f(x,s) \psi(x,s)  \dx \ds.
\]

Gathering the previous relations proves \eqref{eq:pfweakirr}.

\medskip

In order to prove \eqref{eq:pfweakirrz}, that is $u(0)=u_0$, we reproduce the reasoning of section \ref{sec:inihmun}, by proving that the sequence $(u^{(n)})_\nnn$ is relatively compact in $C([0,T],W^{-1,1}_\star(\O))$. Hence, again using Ascoli's Theorem, it is again enough to prove:
\benum
\item For all $t \in [0,T]$, $(u^{(n)}(t))_\nnn$ is relatively compact in $W^{-1,1}_\star(\O)$.
\item $\norm{u^{(n)}(t)-u^{(n)}(s)}{W^{-1,1}_\star(\O)} \to 0$, as $s \to t$, uniformly with respect to $\nnn$ (and for all $t \in [0,T]$).
\eenum
The second item is a consequence of $\partial_t u^{(n)} \in L^{1}(]0,T[,W^{-1,1}_\star(\O))$ since Lemma 4.25 of \cite{edp-gh} gives
for all  $t_1,t_2 \in [0,T]$, $t_1 > t_2$ and all $\nnn$,
\[
u^{(n)}(t_1)-u^{(n)}(t_2)=\int_{t_2}^{t_1} \partial_t u^{(n)}(s)\ds,
\]
and then
\begin{multline*}
\norm{u^{(n)}(t_1)-u^{(n)}(t_2)}{W^{-1,1}_\star(\O)} \le \int_{t_2}^{t_1} \norm{\partial_t u^{(n)}(s)}{W^{-1,1}_\star(\O)} \ds\\
\le \big( \int_0^T \norm{\partial_t u^{(n)}(s)}{W^{-1,1}_\star(\O)}^2\ds\big)^{\frac 1 2}
\sqrt{t_1-t_2}
\\
 \le \norm{\partial_t u^{(n)}}{L^1(]0,T[,W^{-1,1}_\star(\O))} \sqrt{t_1-t_2}.
\end{multline*}
Since, by Lemma \ref{lem:estimwmun}, the sequence $(\partial_t u^{(n)})_\nnn$ is bounded in $L^1(]0,T[,W^{-1,1}_\star(\O))$, one deduces
$\norm{u^{(n)}(t)-u^{(n)}(s)}{W^{-1,1}_\star(\O)} \to 0$, as $s \to t$, uniformly with respect to $\nnn$ (and for all $t \in [0,T]$).

In order to prove the first item, one uses Estimate \eqref{estiml1} (and the fact that $u^{(n)} \in C(]0,T[, H^{-1}(\O)$). 
It gives that 
the sequence $(u^{(n)}(t))_\nnn$ is bounded in $L^1(\O)$  for all $t \in [0,T]$ and then is relatively compact in $W^{-1,1}_\star(\O)$ for all $t \in [0,T]$.
It is then possible to apply Ascoli's Theorem and obtain as it is said before 
$u(0)=u_0$.
This concludes the proof that $u$ is solution of \eqref{eq-pf-weak-irr}.

\appendix
\section{Convergence results with truncations}\label{sec:app}

In the proof of Lemmas \ref{lem:mintyneg} and \ref{lem:mintytk}, we apply the following lemma.
\begin{lemma}\label{lem:cvtkstrong}
Let $E$ be a bounded open set of $\R^N$ ($N \ge 1$). Let $p>1$ be given.
Let $v^{(n)}$ be a bounded sequence of elements of $L^p(E)$ and $v\in L^p(E)$ be such that, for all $k>0$, $T_k(v^{(n)})$ converges to $T_k(v)$ in $L^p(E)$. Then $v^{(n)}$ converges to $v$ in $L^q(E)$ for all $q\in [1,p[$.
\end{lemma}
\begin{proof}
Let us first prove that $v^{(n)}$ converges to $v$ in $L^1(E)$. For any $k >0$, we can write
\begin{multline*}
\int_E |v^{(n)}(x)-v(x)| \dx  \le     \int_E |T_k(v^{(n)})-T_k(v)|  \dx \\ +  \int_E |v^{(n)}-T_k(v^{(n)})|  \dx
+ \int_E |v-T_k(v)|\dx
\\
\le  \int_E |T_k(v^{(n)})-T_k(v)|  \dx  +  \int_{|v^{(n)}| \ge k} |v^{(n)}| \dx + \int_{|v| \ge k} |v|  \dx
\\
 \le     \int_E |T_k(v^{(n)})-T_k(v)|  \dx + \frac 1 {k^{p-1}} \int_E |v^{(n)}|^p \dx+\frac 1 {k^{p-1}} \int_E |v|^p \dx.
\end{multline*}
Let $\eps>0$. Thanks to the $L^p$-bound on $v^{(n)}$ and the fact that $v \in L^p(E)$, there exists $k >0$ such that, for all $n$,
\[
\int_E |v^{(n)}(x)-v(x)| \dx | \le 2 \eps +   \int_E |T_k(v^{(n)})-T_k(v)| \dx.
\]
Then there is $n_0$ such that for $n \ge n_0$,
\[
\int_E |v^{(n)}(x)-v(x)| \dx  \le 3 \eps.
\]
The proof is then complete by interpolation $L^1-L^p$.
\end{proof}

It is interesting to notice that a similar result holds when studying the weak convergence of a sequence whose truncations weakly converge. The proof of this lemma could be obtained from that of Lemma  \ref{lem:cvtkstrong} by using the convergence in the distribution sense. We give hereafter an elementary proof, and a counter-example if $p=1$.

\begin{lemma}\label{lem:cvtk}
Let $E$ be an open set of $\R^N$ ($N \ge 1$). Let $p>1$ be given.
Let $v^{(n)}$ be a bounded sequence of elements of $L^p(E)$ and $v\in L^p(E)$ be such that, for all $k>0$, $T_k(v^{(n)})$ weakly converges to $T_k(v)$ in $L^p(E)$. Then $v^{(n)}$ weakly converges to $v$ in $L^p(E)$.
\end{lemma}
\begin{proof}
For $\vphi \in  L^\infty(E) \cap L^{p'}(E)$ and $k >0$ (we recall that $p'=p/(p-1)$), we have
\begin{multline*}
|\int_E (v^{(n)}(x)-v(x))\vphi (x) \dx | \le    | \int_E (T_k(v^{(n)})-T_k(v)) \vphi \dx|\\ + | \int_E (v^{(n)}-T_k(v^{(n)})) \vphi \dx|
+| \int_E (v-T_k(v)) \vphi \dx|
\\
\le  | \int_E (T_k(v^{(n)})-T_k(v)) \vphi \dx| +  \int_{|v^{(n)}| \ge k} |v^{(n)}|| \vphi| \dx + \int_{|v| \ge k} |v| | | \vphi| \dx
\\
 \le    | \int_E (T_k(v^{(n)})-T_k(v)) \vphi \dx| + \frac 1 {k^{p-1}} \int_E |v^{(n)}|^p| \vphi| \dx+ \frac 1 {k^{p-1}} \int_E |v|^p| \vphi| \dx
\end{multline*}
Let $\eps>0$. Thanks to the $L^p$-bound on $v^{(n)}$, the fact that $v \in L^p(E)$ and $\vphi \in L^\infty(E)$, there exists $k >0$ such that, for all $n$,
\[
|\int_E (v^{(n)}(x)-v(x))\vphi (x) \dx | \le 2 \eps +  | \int_E (T_k(v^{(n)})-T_k(v)) \vphi \dx|.
\]
Then there is $n_0$ such that for $n \ge n_0$ (we use now $\vphi \in L^{p'}(E)$),
\[
|\int_E (v^{(n)}(x)-v(x))\vphi (x) \dx | \le 3 \eps.
\]
Since $L^\infty(E)\cap L^{p'}(E)$ is dense in $L^{p'}(E)$, the proof is complete.
\end{proof}
\begin{remark}
It is interesting to remark that Lemma \ref{lem:cvtk} is false if $L^p$ is replaced by $L^1$.
We give here an example. We take $E=]0,1[$ (with $d=1$)
and for $n \in \mathbb{N}^\star$ we define $v^{(n)}$ as follow:
\begin{align*}
& v^{(n)}(x) = n^3, \textrm{ if } x \in ]\frac i n, \frac i n + \frac 1 {n^2}[, \; i \in \{0,\ldots, n-1\},
\\
& v^{(n)}(x) = 0 \textrm{ elsewhere}.
\end{align*}
With this choice, one has $T_k(v^{(n)}) \to 0=T_k(0)$ in $L^1(E)$ but $v^{(n)} \to 1_E$ weakly in $L^1(E)$, as $\nti$.
\end{remark}

\end{document}